\documentclass[11pt,hyp]{nyjm}
\usepackage{hyperref}
\hypersetup{nesting=true,debug=true,naturalnames=true}
\usepackage{graphicx,amssymb,upref}

\let\<\langle
\let\>\rangle
\newcommand{\doi}[1]{doi:\,\href{http://dx.doi.org/#1}{#1}}

\let\uml\"

\title{Automatic Continuity of Pure Mapping Class Groups}

\author{Ryan Dickmann}  
\address{686 Cherry St NW, Atlanta, GA 30332} 
\email{rdickmann3@gatech.edu}  
\thanks{The author acknowledges support from NSF grant RTG DMS--1745583.} % \thanks entries are to acknowledge grants. 



\usepackage{url}
\usepackage{enumerate}
\usepackage[mathscr]{euscript}
\usepackage{mathtools}
\usepackage{tikz-cd}
\usepackage{tikz}
\usetikzlibrary{arrows,chains,matrix,positioning,scopes}
\usetikzlibrary{shapes.geometric}

\newcommand{\p}[1]{\medskip \noindent \emph{#1}.}

\newcommand{\R}{\mathbb{R}}
\newcommand{\Z}{\mathbb{Z}}

\newcommand{\Q}{\mathbb{Q}}

\newcommand{\N}{\mathbb{N}}

\newcommand{\MCG}{\operatorname{Map}}
\newcommand{\PMCG}{\operatorname{PMap}}
\newcommand{\PMCGc}{\PMCG_{c}}
\newcommand{\PMCGcc}[1]{\overline{\PMCGc(#1)}}
\newcommand{\Homeo}{\operatorname{Homeo}}

\newcommand{\supp}{\operatorname{supp}}

\newcommand{\Stab}{\operatorname{Stab}}

\newcommand{\lilo}{\mathrm{o}}

\let\emptyset\varnothing

\newtheorem*{theorem_a}{Theorem A}
\newtheorem*{theorem_b}{Theorem B}

\newtheorem{theorem}{Theorem}[section]
\newtheorem{proposition}[theorem]{Proposition}
\newtheorem{corollary}[theorem]{Corollary}
\newtheorem{lemma}[theorem]{Lemma}

\theoremstyle{definition}

\newtheorem{remark}[theorem]{Remark}

\begin{document}

\begin{abstract}
We completely classify the orientable infinite-type surfaces $S$ such that $\PMCG(S)$, the pure mapping class group, has automatic continuity. This classification includes surfaces with noncompact boundary. In the case of surfaces with finitely many ends and no noncompact boundary components, we prove the mapping class group $\MCG(S)$ does not have automatic continuity. We also completely classify the surfaces such that $\PMCGcc{S}$, the subgroup of the pure mapping class group composed of elements with representatives that can be approximated by compactly supported homeomorphisms, has automatic continuity. In some cases when $\PMCGcc{S}$ has automatic continuity, we show any homomorphism from $\PMCGcc{S}$ to a countable group is trivial. 
\end{abstract}

\maketitle

%%%%%%%%%%%%%
%%%%%%%%%%%%%%%%%%%
\section{Introduction} \label{intro}

A surface will refer to a second-countable, connected, orientable, 2-manifold, possibly with boundary. Let $\Homeo_\partial(S)$ be the group of (orientation-preserving) homeomorphisms of $S$ that fix the boundary pointwise. The \textit{mapping class group} $\MCG(S)$ is defined to be \begin{align*}
    \MCG(S) = \Homeo_{\partial}(S)/\sim
\end{align*} where two homeomorphisms are equivalent if they are isotopic relative to the boundary of $S$. A \textit{degenerate end} will refer to an end with a closed neighborhood homeomorphic to a disk with boundary points removed. Throughout the paper, we assume surfaces do not have degenerate ends, since filling in degenerate ends does not change the underlying mapping class group.

A surface is said to be of \textit{infinite type} when $\pi_1$ is infinitely generated, otherwise, it is of \textit{finite type}. A \textit{Polish group} is a topological group that is separable and completely metrizable. In the finite-type case, mapping class groups of surfaces are finitely generated and are therefore countable with no interesting Polish group structure. Mapping class groups for infinite-type surfaces, however, are uncountable and are Polish groups when given the quotient topology inherited from the compact-open topology on $\Homeo_{\partial}(S)$.

Mann \cite{Mann2019} proved that certain mapping class groups of infinite-type surfaces without boundary have \textit{automatic continuity}; i.e., every homomorphism from these groups to a separable group is continuous. Mann also found examples of mapping class groups that admit discontinuous homomorphisms to a finite group and asked which mapping class groups have automatic continuity. Towards this question, we fully classify the pure mapping class groups that have automatic continuity. 

\p{Pure mapping class groups} The \textit{pure mapping class group} of a surface, denoted $\PMCG(S)$, is the subgroup of the mapping class group consisting of elements that fix the ends of the surface. A \textit{disk with handles} will refer to any surface that can be constructed by taking a disk, removing a closed, totally disconnected set from the boundary (whose points become the ends of the surface), and then attaching infinitely many handles accumulating to some subset of the ends. See Figure \ref{DWH} for an example. The assumption of infinitely many handles is simply to rule out finite-type cases.

\begin{theorem_a} \label{thma}
Let $S$ be an infinite-type surface. Then $\PMCG(S)$ has automatic continuity if and only if 

\begin{enumerate} [(i)]
    \item $S$ is a connected sum of finitely many disks with handles with any finite-type surface, and
    \item $S$ has finitely many ends accumulated by genus.
\end{enumerate}
\end{theorem_a}

The finite-type surface is necessary in the first condition to capture additional cases with finitely many compact boundary components and finitely many punctures. The final condition is required since, for surfaces with infinitely many ends accumulated by genus, we show there is a discontinuous homomorphism $\PMCG(S) \rightarrow \Z_2$ which factors through $\Z^\omega$, the infinite countable product. If we equip $\Z$ with the discrete topology, then $\Z^\omega$ is a Polish group with the product topology. The map to $\Z^\omega$ is given by the work of Aramayona--Patel--Vlamis in the compact boundary case \cite{APV2017}, and this was extended by the author to the noncompact boundary case \cite{Dickmann_2023}. More precisely, their works show $\PMCG(S) = \PMCGcc{S}$ when $S$ has at most one end accumulated by genus, and otherwise $\PMCG(S)$ factors into a semidirect product of a special subgroup with $\Z^n$ where $n$ is finite if and only if there are finitely many ends accumulated by genus. We now discuss  $\PMCGcc{S}$ further.

\begin{figure}[!htbp]
\centering
\includegraphics[width = 0.6\textwidth]{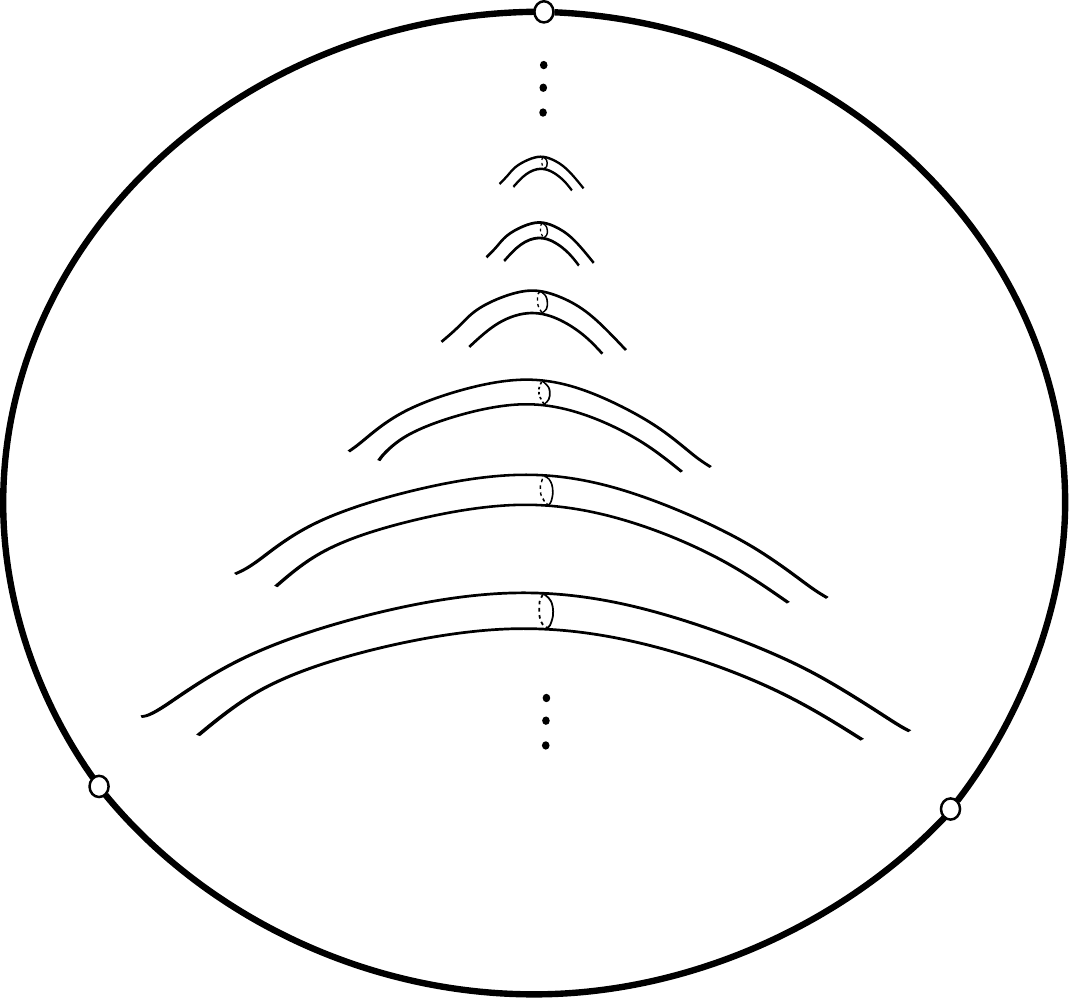}
\caption{A visualization of a disk with handles with two ends. The handle attaching procedure in this case joins together the bottom two ends of the disk into a single end. By Theorem \hyperref[thma]{A}, $\PMCG(S)$ has automatic continuity for this surface. Note $\MCG(S)=\PMCG(S)$ for any disk with handles since fixing the boundary forces the ends to be fixed.}
\label{DWH}
\end{figure}

\p{Closure of the subgroup of compactly supported maps} We say $f \in \MCG(S)$ is \textit{compactly supported} if $f$ has a representative that is the identity outside of a compact subset of $S$. The subgroup consisting of compactly supported mapping classes is denoted $\PMCGc(S)$ since every compactly supported mapping class is pure. The closure of this subgroup, denoted $\PMCGcc{S}$, can be described as the subgroup composed of elements with representatives that can be approximated by compactly supported homeomorphisms. We also fully classify the $\PMCGcc{S}$ that have automatic continuity.

\begin{theorem_b} \label{thmb}
Let $S$ be an infinite-type surface. Then $\PMCGcc{S}$ has automatic continuity if and only if $S$ is a connected sum of finitely many disks with handles with any finite-type surface.
\end{theorem_b}

As a consequence of Mann's result \cite{Mann2019} of automatic continuity for the mapping class groups of the sphere minus the Cantor set and the plane minus the Cantor set, Vlamis \cite{vlamis2020perfect} showed that any homomorphism from these groups to a countable group is trivial. Using a similar but independent proof, we show the following.

\begin{corollary}\label{cor1.1}
   Let $S$ be a disk with handles. Then every homomorphism from $\PMCGcc{S}$ to a countable group is trivial. Therefore, $\PMCGcc{S}$ contains no proper normal subgroups of countable index and no proper subgroups of finite index. 
\end{corollary}

This is in stark contrast to the mapping class groups of finite-type surfaces which are residually finite. One natural approach to studying infinite groups is to investigate their finite quotients, but for the $\PMCGcc{S}$ of disks with handles, we do not even have countable quotients to work with. Note that $\PMCG(S)$ always has a proper normal subgroup of countable index when $S$ has at least two ends accumulated by genus. In particular, when there are finitely many ends accumulated by genus, $\PMCGcc{S}$ is the desired subgroup, and when there are infinitely many ends accumulated by genus, the kernel of the discontinuous homomorphism to $\Z_2$ discussed above is the desired subgroup.

\p{Mapping class groups} Using the same techniques in the proofs of the above theorems, we are also able to comment on the automatic continuity of the full mapping class groups.   

\begin{theorem} \label{thm1.2}
        Suppose $S$ is an infinite-type surface with finitely many ends and no noncompact boundary components. Then $\MCG(S)$ does not have automatic continuity. 
\end{theorem}

For example, the mapping class group of the \textit{ladder surface}, the unique surface with no boundary and exactly two ends each accumulated by genus, does not have automatic continuity. We also extend the reverse direction of Theorem \hyperlink{thma}{A} to the full mapping class group. 

\begin{theorem} \label{thm1.3}
    Suppose $S$ is an infinite-type surface satisfying the conditions of Theorem \hyperref[thma]{A}. Then $\MCG(S)$ has automatic continuity.
\end{theorem}

\p{Outline} In Section \ref{noncompact}, we discuss some background on surfaces with noncompact boundary, and in Section \ref{tools}, we discuss the tools needed to prove the reverse directions of Theorems \hyperref[thma]{A} and \hyperref[thmb]{B}. In Section \ref{reverse}, we prove the reverse directions as well as Theorem \ref{thm1.3} using adaptations Mann's techniques \cite{Mann2019} and tools from the author for working with surfaces with noncompact boundary \cite{Dickmann_2023}. We also use a new extension to a classical lemma of Sierpi\'nski; see Section \ref{sierp}. In Section \ref{forward}, we prove the forward directions of Theorems \hyperref[thma]{A} and \hyperref[thmb]{B}, and in Section \ref{extra}, we prove Corollary \ref{cor1.1} and Theorem \ref{thm1.2}.

Automatic continuity proofs largely rely on some form of self-similarity in a given group, and in particular, we take advantage of the self-similarity of the mapping class groups of the sliced Loch Ness monsters. A \textit{sliced Loch Ness monster} is any surface with nonempty boundary, no compact boundary components, infinite genus, and one end. See Figure \ref{SLNM}. The key idea is that any sliced Loch Ness monster contains closed proper copies of itself, and therefore, the mapping class group does as well; see Section \ref{moieties}. On the other hand, the \textit{Loch Ness monster}, the unique surface with one end, infinite genus, and empty boundary, does not contain a closed proper copy of itself,\footnote{It is unknown whether the mapping class group of the Loch Ness monster contains a proper copy of itself. A group that does not contain a proper copy of itself is known as \textit{co-Hopfian}. Aramayona--Leininger--McLeay \cite{aramayona2021big} have studied the co-Hopfian property for mapping class groups of infinite-type surfaces, and in particular they found uncountably many examples of pure mapping class groups that are not co-Hopfian.} and Domat and the author showed its mapping class group does not have automatic continuity \cite{domat2020big}. 

Once we have found examples of surfaces such that $\PMCG(S)$ and $\PMCGcc{S}$ have automatic continuity, the main difficulty in proving Theorems \hyperref[thma]{A} and \hyperref[thmb]{B} is ruling out the zoo of remaining surfaces. Note these theorems consider all surfaces including those with complicated end spaces such as large countable ordinals. Using the tools developed by the author for decomposing surfaces into simpler pieces, we can reduce the complexity of the problem significantly.

\begin{figure}[!htbp]
\centering
\includegraphics[width = 0.5\textwidth]{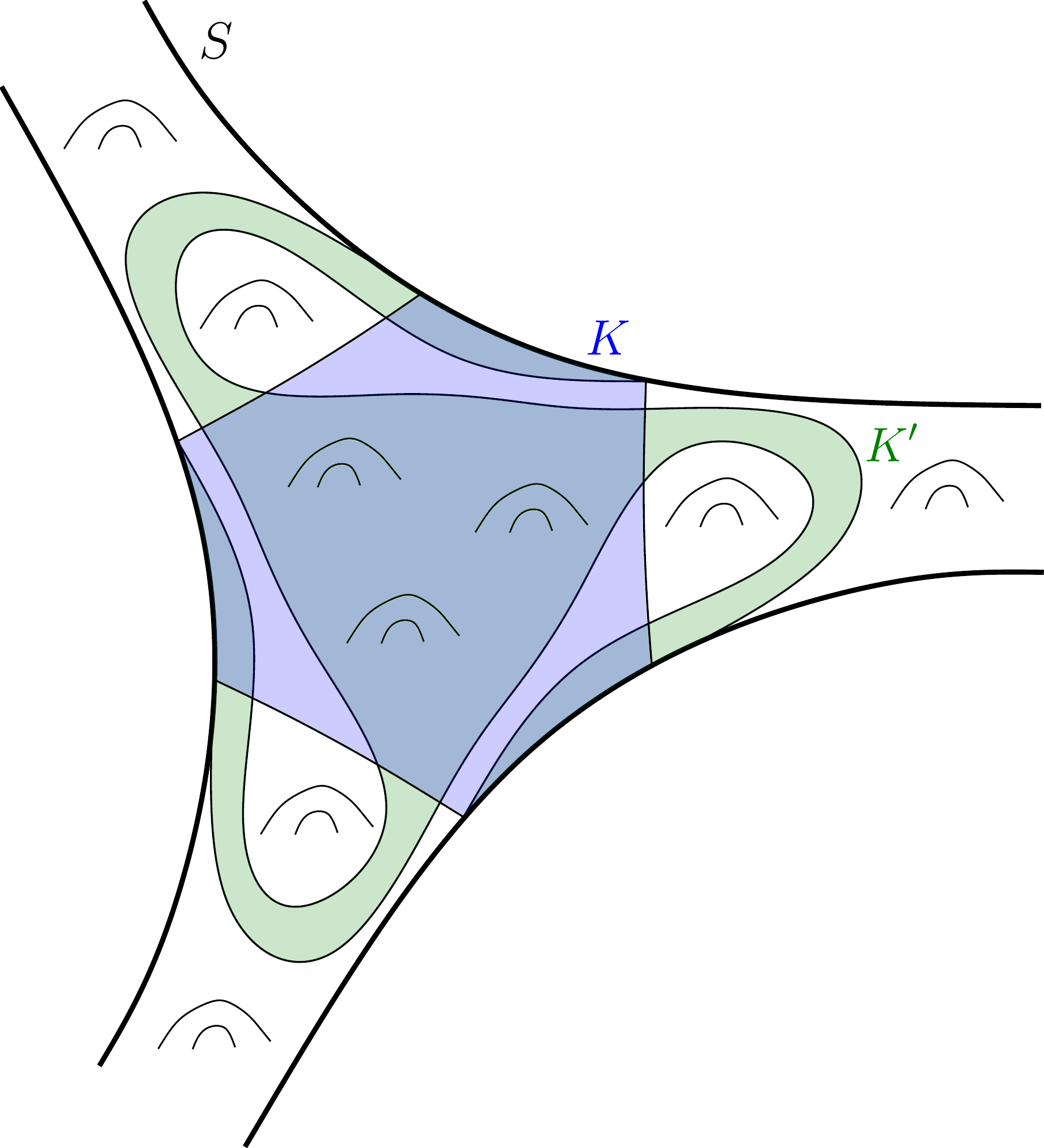}
\caption{Two subsurfaces $K$ and $K^\prime$ in the same $\PMCGcc{S}$ orbit.}
\label{badintersection}
\end{figure}

To find discontinuous homomorphisms in the remaining cases, we use the work of Domat \cite{domat2020big} who showed $\PMCGcc{S}$ admits uncountably many discontinuous homomorphisms to $\Q$ with the discrete topology when $S$ contains a certain infinite sequence of nondisplaceable subsurfaces. A \textit{nondisplaceable} subsurface in this case refers to a surface that cannot be mapped off of itself by any map in $\PMCGcc{S}$. Domat's proof relies on subsurface projections of Masur--Minksy \cite{masur1998geometry} to construct projection complexes of Bestvina--Bromberg--Fujiwara \cite{BBF2015}. For a given finite-type nondisplaceable subsurface $K$, a projection complex is built using subsurface projections between subsurfaces in the orbit of $K$ under the action of $\PMCGcc{S}$ on the isotopy classes of subsurfaces in $S$. The subsurface projection to $K$ is a map that takes a subsurface $K^\prime$ (distinct from $K$) in the orbit of $K$ and outputs an element of the power set of the vertex set of the curve graph of $K$. The vertex set of the curve graph is defined as the set of isotopy classes of essential simple closed curves. Recall a curve is trivial when it bounds a disk, peripheral when it bounds an annulus, and essential when it is neither trivial nor peripheral. The subsurface projection to $K$ is defined using the fact that $\partial K^\prime \cap K$ is a collection of curves and arcs in $K$. The arcs are turned into curves by surgering on intervals in $\partial K$. The issue for surfaces with noncompact boundary is that there can exist nondisplaceable subsurfaces $K$ and $K^\prime$ in the same $\PMCGcc{S}$ orbit such that $\partial K^\prime \cap K$ is the union of trivial arcs in $K$ (a trivial arc is one that bounds a disk). Trivial arcs yield trivial or peripheral curves after surgery using the boundary of $K$, so the subsurface projection is not well-defined. See Figure \ref{badintersection} for an example. We will see that the surfaces from Theorem \hyperref[thmb]{B} are exactly those that do not have the special sequences of nondisplaceable subsurfaces needed by Domat.

 \subsection*{Acknowledgments} The author would like to thank Kathryn Mann for introducing him to automatic continuity and the many techniques used in Section \ref{tools}. Thank you to the organizers of the 2019 AIM workshop for surfaces of infinite type. Thank you to Dan Margalit, Roberta Shapiro, and Sanghoon Kwak for comments on an earlier draft. Thank you to an anonymous referee for carefully reading this paper.  

%%%%%%%%%%%%%
%%%%%%%%%%%%%%%%%%%
\section{Surfaces with Noncompact Boundary} \label{noncompact}

Here we discuss some background on surfaces with noncompact boundary needed for the proofs of Theorems \hyperref[thma]{A} and \hyperref[thmb]{B}. We will assume the reader is familiar with the Richards classification of infinite-type surfaces without boundary \cite{Richards1963} as well as the definition of the ends space of a surface, planar ends, and ends accumulated by genus. These definitions apply without adaptation to surfaces with noncompact boundary. The first adaptation needed for noncompact boundary is that we must consider \textit{ends accumulated by compact boundary}, ends for which every closed neighborhood contains infinitely many compact boundary components, then we must consider the noncompact boundary components.

\p{Boundary chains} Deleting the noncompact boundary components of a surface induces a map $\pi$ from the ends space of the surface to the ends space of the interior surface (for more details see \cite{Dickmann_2023}, Section 4.2). For example, in Figure \ref{SLNM} deleting the noncompact boundary components induces a map sending the two ends to the single end of the interior. Suppose $e$ is an end of the surface that a noncompact boundary component points to, and let $e^\lilo = \pi(e)$ be the corresponding end of the interior surface. Then we refer to $\pi^{-1}(e^\lilo)$ as a \textit{boundary chain}, and we refer to any end in $\pi^{-1}(e^\lilo)$ as a \textit{boundary end}. Note any disk with handles has a single boundary chain, and every end is a boundary end. Other examples of surfaces with a single boundary chain can be constructed by taking a disk, deleting a set from the boundary, and then attaching surfaces without noncompact boundary components where these attached surfaces may be infinite-type and may accumulate to the set of deleted points. An end that is not a boundary end will be called an \textit{interior end}. Though the boundary chain is formally defined as a set of ends, we can also think of a boundary chain as the corresponding union of noncompact boundary components.

\p{Brown--Messer classification of surfaces} The classification of infinite-type surfaces with boundary is due to Brown and Messer \cite{classification}.  Roughly speaking, their theorem states that surfaces with boundary are classified up to homeomorphism by the Richards classification data, the ends accumulated by compact boundary, and additional data describing the boundary chains. The major achievement of Brown and Messer was finding a way to represent this boundary chain data, though it is fairly technical. Thus we will not state the actual classification theorem, and instead, we will use tools of the author developed for working with surfaces with boundary \cite{Dickmann_2023}. More examples of surfaces with boundary can be found in Section 3 of the previous paper of the author. 

Recall a sliced Loch Ness monster is any surface with nonempty boundary, no compact boundary components, infinite genus, and one end. One immediate application of the classification of surfaces is that a sliced Loch Ness monster is determined by the number of boundary components. We will refer to an $n$-sliced Loch Ness monster to emphasize the number of boundary components. Note that $n$ may be infinite.

Recall a disk with handles is a surface that can be constructed by taking a disk, removing a closed, totally disconnected set from the boundary, and then attaching infinitely many handles accumulating to some subset of the ends. Sliced Loch Ness monsters are examples of disks with handles since we can construct any sliced Loch Ness monster by attaching handles to a disk with points removed from the boundary in a way that joins every end to a single end. 

\begin{remark}
    We could have equivalently defined a sliced Loch Ness monster as a disk with handles with exactly one end. There are other constructions of sliced Loch Ness monsters that do not start with a single disk which will be useful for the proof of Theorem \hyperref[thma]{A}. We will discuss these in Section \ref{moieties}. See Figure \ref{SLNM} for some examples. 
\end{remark}

\begin{figure}[!htbp]
\centering
\includegraphics[width = \textwidth]{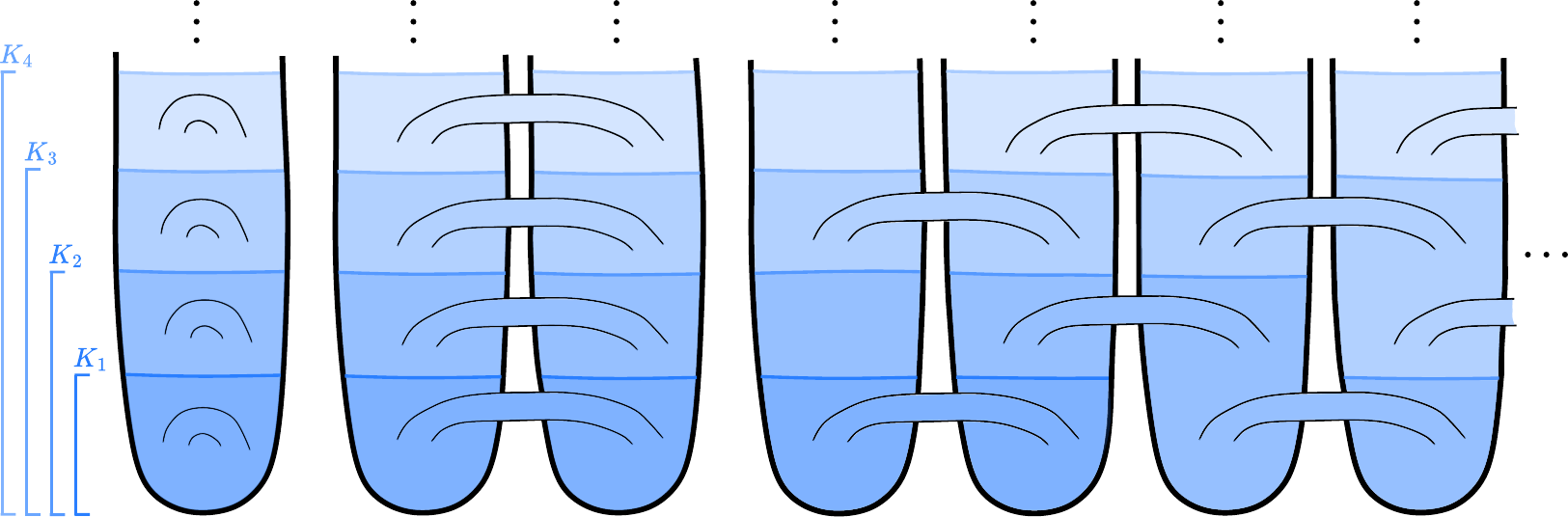}
\caption{A 1-sliced, a 2-sliced, and an $\infty$-sliced Loch Ness monster with the noncompact boundary components represented by the bold lines. The blue shading represents a given compact exhaustion $\{K_i\}$ for each surface.}
\label{SLNM}
\end{figure}

\p{Cutting up a surface with boundary} The following result of the author \cite{Dickmann_2023} shows that any surface with only ends accumulated by genus can be decomposed in some sense into sliced Loch Ness monsters and Loch Ness monsters.

\begin{lemma} \label{fulldecomp}
Every disk with handles with every end accumulated by genus can be cut along a collection of disjoint essential arcs into sliced Loch Ness monsters.

Furthermore, any infinite-type surface with every end accumulated by genus can be cut along disjoint separating curves into components that are either            
\begin{enumerate}[(i)]
    \item Loch Ness monsters with compact boundary components added, or
    \item disks with handles with compact boundary components added. 
\end{enumerate}

\end{lemma}

Cutting a surface $S$ along a curve or arc $\alpha$ yields a possibly disconnected surface with an identification map between subsets of the boundary such that the quotient is $S$ and the image of the identified subsets under the quotient is $\alpha$. When we say compact boundary components are added, we mean open balls with disjoint closures are removed. These components may have any number of compact boundary components added, and if we add infinitely many we assume they accumulate to some end of the original surface. 

Recall a planar end is simply one that is not accumulated by genus. Due to the assumption on degenerate ends in the introduction, a disk with handles automatically has no planar ends, so we can apply the first part of Lemma \ref{fulldecomp} to any disk with handles. Note the second part of Lemma \ref{fulldecomp} does not immediately extend to surfaces with planar ends since filling in a planar boundary end may not be possible; for example, if it is accumulated by compact boundary components or accumulated by boundary chains. For other decomposition results concerning general surfaces, see Section 4 of the work of the author \cite{Dickmann_2023}. 

We will need the following result to justify the forward directions of Theorems \hyperref[thma]{A} and \hyperref[thmb]{B}.  

\begin{lemma} \label{finitechain}
    An infinite-type surface with nonempty boundary, finitely many boundary chains, no compact boundary components, and no interior ends is a connected sum of disks with handles.
\end{lemma}

\begin{proof}
    By the assumption on interior ends, the only planar ends must be boundary ends. Since there are no compact boundary components and finitely many boundary chains, these planar ends must be degenerate. Since we assumed in the introduction that surfaces do not have degenerate ends, every end must be accumulated by genus and we can apply the second part of Lemma \ref{fulldecomp} to cut the surface along curves into disks with handles with compact boundary components added and Loch Ness monsters with compact boundary added. Note none of the components can be the second type since then there would be interior ends, so we are done. 
\end{proof}

\subsection{Standard pieces of sliced Loch Ness monsters} \label{moieties}

Now we discuss certain models of the sliced Loch Ness monsters and a standard way to break them into self-similar pieces. We represent the 1-sliced Loch Ness monster as a closed upper half-plane with a handle attached in a small ball about each integer point. Let $M_i$ be the subsurface bounded by the lines $x=i-\frac{1}{2}$ and $x=i+\frac{1}{2}$ for $i \in \Z$. We refer to each $M_i$ as a \textit{standard piece} of the 1-sliced Loch Ness monster.

\begin{figure} [!htbp]
\centering
\includegraphics[width = \textwidth]{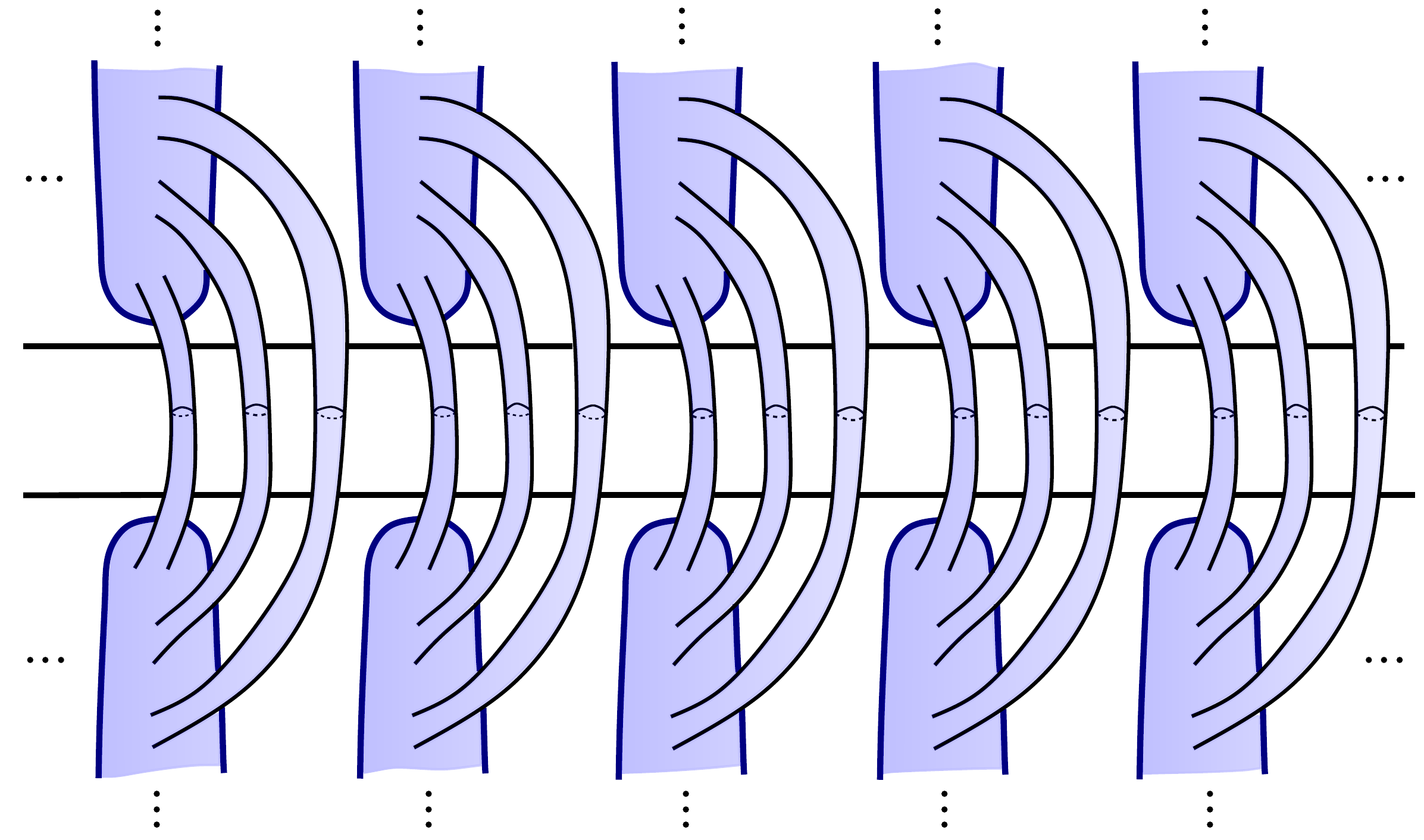}
\caption{Shown in the shaded regions are disjoint representatives of the standard pieces for the 2-sliced Loch Ness monster.}
\label{2sliced}
\end{figure}

Let $K_n$ denote a compact surface with zero genus and $n$ boundary components when $n$ is finite. Let $K_\infty$ denote the surface with 0 genus, no noncompact boundary components, and exactly two ends that are accumulated by compact boundary components. To construct an $n$-sliced Loch Ness monster with $n \geq 2$ and possibly infinite, first start with a disjoint union of $n$ copies of the closed upper half-plane. Now for each interior integer point in the closed upper half-plane apply the following construction: first, remove a small open ball about the integer point in each of the $n$ copies, then take a copy of $K_n$ and attach $\partial K_n$ to the resulting boundary components. When $n$ is infinite, we arrange the closed upper half-planes according to a $\Z$-index. We then equip the components of $\partial K_\infty$ with the natural $\Z$-index and apply the attaching procedure respecting the two indexes.

To construct the standard pieces in these cases, let $M_i$ be the subsurface bounded by the $n$ copies of the lines $x=i-\frac{1}{2}$ and $x=i+\frac{1}{2}$ for $i \in \Z$. See Figure \ref{2sliced} for an example where a small open neighborhood of the boundary has been removed from each $M_i$ to emphasize the standard pieces.

%%%%%%%%%%%%%
%%%%%%%%%%%%%%%%%%%
\section{Tools} \label{tools}

We now discuss the various tools needed in Section \ref{reverse} to show that $\PMCG(S)$ and $\PMCGcc{S}$ have automatic continuity for certain surfaces with noncompact boundary. These groups are related by the work of Aramayona--Patel--Vlamis in the compact boundary case \cite{APV2017}, and this was extended by the author to the noncompact boundary case \cite{Dickmann_2023}:

\begin{lemma} \label{structure}
 Let $S$ be an infinite-type surface. $$\PMCG(S) = \overline{\PMCG_c(S)} \rtimes H$$ where $\displaystyle H \cong \Z^{n-1}$ when there is a finite number $n > 1$ of ends of $S$ accumulated by genus, $\displaystyle H \cong \Z^\omega$ when there are infinitely many ends accumulated by genus, and $H$ trivial otherwise. 
\end{lemma}

Note this implies $\PMCGcc{S} = \PMCG(S)$ for a one-ended surface such as a sliced Loch Ness monster. Since $\PMCG(S)=\MCG(S)$ in this case as well, we will always use the latter notation for one-ended surfaces or subsurfaces.

\subsection{Automatic Continuity} \label{autocont}
The standard approach to proving automatic continuity is to prove a stronger but more tractable condition due to Rosendal--Solecki \cite{RS2007}. We say a subset of a group is \textit{countably syndetic} if countably many left translates cover the entire group.  A topological group is \textit{Steinhaus} if there exists an integer $k$ such that, for every countably syndetic symmetric subset $W$ of $G$, $W^k$ contains an open neighborhood of the identity. 

\begin{proposition} [Rosendal--Solecki] \label{rosendal}
A Steinhaus Polish group has automatic continuity.
\end{proposition}          

We will also need a common fact used in automatic continuity proofs. This result, as well as the result of Rosendal and Solecki, follows from the Baire category theorem.

\begin{proposition} \label{bct}
    Let $G$ be a Polish group and $W \subset G$ a countably syndetic symmetric set. Then there exists a neighborhood $U$ of the identity in $G$ such that $W^2$ is dense in $U$.
\end{proposition} 

In some cases, we will rule out automatic continuity by ruling out a weaker property. A topological group is said to have the \textit{small index property} when any countable index subgroup is open. The following is well-known.
 
 \begin{proposition} \label{smallindex}
 A Polish group that has automatic continuity has the small index property.
 \end{proposition}
 
 \begin{proof}
 Let $S_\omega$ denote the symmetric group on a countably infinite set. This is a Polish group with the compact-open topology. Any countable index subgroup $H$ determines a homomorphism $\phi$ to $S_\omega$ by the left multiplication action on left cosets. The subset of $S_\omega$ corresponding to permutations that fix $H$ is open. The pullback of this subset via $\phi$ is $H$, which is open by automatic continuity.  
 \end{proof}

\subsection{Sierpi\'nski lemmas} \label{sierp}

To apply the same techniques used by Mann \cite{Mann2019}, we need the following result of Sierpi\'nski.

\begin{lemma} [Sierpi\'nski] \label{sierpinski}
For an infinite countable set $\Lambda$, there is an uncountable collection of infinite subsets $\{\Omega_\alpha\}_{\alpha \in \Gamma}$ of $\Lambda$ such that any $\Omega_\alpha \cap \Omega_\beta$ is finite for $\alpha \neq \beta$.
\end{lemma}

\begin{proof}
Identify $\Lambda$ with $\Q$ via a bijection, and let $\Gamma = \R \setminus \Q$. For any given $\alpha \in \Gamma$, let $\Omega_\alpha$ be any sequence of rational numbers converging to $\alpha$. 
\end{proof}

We also need an extension that allows us to apply the techniques to infinite unions. 
 
\begin{lemma} \label{sierpinski2}
For any product of infinite countable sets $\Lambda \times \Lambda^\prime$, there is an uncountable collection of infinite subsets $\{\Omega_\alpha\}_{\alpha \in \Gamma}$ of  $\Lambda \times \Lambda^\prime$ such that

\begin{enumerate}  [(i)]
    \item $\Omega_\alpha \cap \Omega_\beta$ is finite for all $\alpha \neq \beta$.
    \item $\Omega_\alpha \cap (\{\lambda\} \times \Lambda^\prime)$ is infinite for all $\lambda \in \Lambda$.
\end{enumerate} 
\end{lemma}

\begin{proof}
Let $\gamma_1, \gamma_2, ...$ be an infinite sequence of irrationals independent over $\Q$. Now $\{\gamma_n \Q\}_{n=1}^\infty$ is a collection of pairwise disjoint dense subsets of the reals. Identify $\Lambda$ via a bijection with $\N$, and then identify each $\{n\} \times \Lambda^\prime$ by a bijection with $\gamma_n \Q$. Let $\Gamma = \R \setminus (\bigcup _{n=1}^\infty \gamma_n \Q)$. For any given $\alpha \in \Gamma$, choose any sequence of numbers in $\bigcup_{n=1}^\infty \gamma_n \Q$ converging to $\alpha$ that includes infinitely many entries of $\gamma_n \Q$ for all $n$, and then let $\Omega_{\alpha}$ be the corresponding set of tuples $(n, \gamma_n q_n)$ where $q_n \in \Q$.
\end{proof}

\subsection{Tools of Mann} \label{mann}

Using the previously discussed tools, we will now introduce the main lemmas for our automatic continuity proofs, Lemmas \ref{cofinitelem} and \ref{cofinitelem2}. The proofs of these lemmas follow an argument from Section 4 of the paper of Mann \cite{Mann2019}, and rely on a modified version of Lemma 3.2 from the same paper. The original lemma was applied to the homeomorphism group of a manifold, but the proof can be adapted to work in the mapping class group setting. 

We say a collection of disjoint subsurfaces $\{S_i\}$ of $S$ is \textit{admissible} when any product $\prod_i f_i$ is a well-defined homeomorphism of $S$ for $f_i$ supported in $S_i$. This condition is required in the statement of the following lemma since Mann's proof uses infinite products of homeomorphisms, and these may not always be well-defined.  For example, consider a sequence of disjoint essential annuli which all essentially intersect some compact subsurface. An infinite product of Dehn twists about these annuli is not a well-defined mapping class of the surface since, even after isotopy, the annuli accumulate at some point of the compact subsurface, and there is no continuous extension of the infinite twist to this point.  

\begin{lemma} [Mann] \label{techlem}
Let $S$ be an infinite-type surface, and $W \subset \PMCGcc{S}$ a countably syndetic symmetric set. Let $\mathcal{A}$ be an infinite admissible collection of disjoint closed subsurfaces of $S$ satisfying: 
\begin{enumerate}
\item There exists an infinite admissible collection of disjoint subsurfaces $U_i$ of $S$ such that each $U_i$ contains an infinite family of  disjoint subsurfaces belonging to $\mathcal{A}$. 
\item There exists $p \in \N$ such that, for each $A \in \mathcal{A}$, the subgroup of $\PMCGcc{S}$ consisting of maps with support in $A$, denoted $\PMCG(A)$, has commutator length bounded by $p$. 
\end{enumerate}
Then there exists $A \in \mathcal{A}$ such that $\PMCG(A)$ is contained in $W^{8p}$.  
\end{lemma}

We will also use the following result of the author \cite{Dickmann_2023}.

\begin{lemma}\label{uniformperfect}
Let $S$ be a disk with handles. Then every element in $\PMCGcc{S}$ can be written as the product of two commutators. 
\end{lemma}

\begin{lemma} \label{cofinitelem}
Let $S$ be a surface containing a subsurface $M$ homeomorphic to any sliced Loch Ness monster, $\{M_n\}$ the standard pieces of this sliced Loch Ness monster, and $W \subset \PMCGcc{S}$ a countably syndetic symmetric set. Then there is a finite set $F$ such that $$\displaystyle\prod_{n \notin F} \MCG(M_n) \subset W^{80}$$
\end{lemma}

\begin{proof}
Apply Lemma \ref{techlem} with $\mathcal{A}$ consisting of subsurfaces of the form $A_{\Lambda} = \cup_{n \in \Lambda} M_{2n}$ for some infinite set $\Lambda \subset \Z$. Any collection of the $M_i$ is admissible since we can extend a homeomorphism on the collection via the identity map to $S$. Note $\mathcal{A}$ satisfies the hypotheses of the lemma, since 

\begin{enumerate} [(i)]
    \item We can write $\Z$ as a countable disjoint union of infinite sets $\Lambda_i$, and define $U_i$ to be $\bigcup_{n \in \Lambda_i} M_{2n}$. Each such set contains a countable union of disjoint elements of $\mathcal{A}$.
    \item Lemma \ref{uniformperfect} implies the same statement for the mapping class group of a disjoint union of sliced Loch Ness monsters. Thus, any element supported in $A_{\Lambda}$ may be written as the product of two commutators.
    
\end{enumerate}

We conclude that for some such subsurface $A_{\Lambda} \in \mathcal{A}$, we have $\MCG(A_{\Lambda}) \subset W^{16}$. Now we apply Lemma \ref{sierpinski}. Let $\{\Omega_\alpha\}$ be an uncountable collection of infinite subsets of $\Lambda$ such that $\Omega_\alpha \cap \Omega_\beta$ is finite for all $\alpha \neq \beta$. Note we may assume $\Lambda$ and each  $\Omega_\alpha$ contain infinitely many negative and positive integers.

Since all homeomorphisms are assumed to fix the boundary pointwise, we first modify each $M_i$ by deleting a small regular open neighborhood of the $x$-axis so that we can move them into one another with an appropriate homeomorphism. For each $\alpha$, let $f_\alpha$ be a homeomorphism supported in $M$ with the following property. For each $n \in \Omega_\alpha$, let $f_\alpha(M _{2n})$ be the smallest connected subsurface containing the union of $M_{2n+1}, M_{2n+2},..., M_{2k-1}$ where $k \in \Omega_\alpha$ is the smallest element in $\Omega_\alpha$ larger than $n$, so that $f_\alpha$ maps $A_{\Omega_\alpha}$ into the complementary region. Also let $f_\alpha$ map the union of $M_{2n+1}, M_{2n+2},..., M_{2k-1}$ into $M_{2k}$. Note this homeomorphism exists by the change of coordinates principle. Since $\{\Omega_\alpha\}$ is uncountable, there are some $\alpha$ and $\beta$ such that $f_\alpha$ and $f_\beta$ are in the same left translate $gW$ for some $g \in \PMCGcc{S}$. Therefore, $f_\alpha^{-1}f_\beta$ and $f_\beta^{-1}f_\alpha$ are both in $W^2$.

 If $n \notin \Omega_\alpha$, then $f_\alpha(M_{2n}) \subseteq M_{2m}$ for some $m \in \Omega_\alpha$. If $m \notin \Omega_\beta$, then $f_\beta^{-1}f_\alpha(M_{2n})$ is contained in some $M_{2k}$ where $k \in \Omega_\beta$. Since $\Omega_\alpha \cap \Omega_\beta$ is finite, we conclude that, with the exception of finitely many values of $n \notin \Omega_\alpha$, the map $f_\beta^{-1}f_\alpha$ takes $M_{2n}$ into $A_{\Omega_\beta} \subset A_\Lambda$.

Reversing the role of $\alpha$ and $\beta$, the same argument shows that with only finitely many exceptions of $n \notin \Omega_\beta$, $f^{-1}_\alpha f_\beta$ takes every $M_{2n}$ into $A_\Lambda$. Let $F'$ be the union of these two exceptional sets of integers. Now write $\bigcup_{n \in \Z}M_{2n}$ as the union of $X_1 = \bigcup_{n \notin (\Omega_\alpha \cup F')} M_{2n}$, $X_2 = \bigcup_{n \notin (\Omega_\beta \cup F')} M_{2n}$, and $X_3 = \bigcup_{n \in F'} M_{2n}$.
\begin{gather*}
f_\beta^{-1}f_\alpha \MCG(X_1) (f_\beta^{-1}f_\alpha)^{-1} \subset \MCG(A_\Lambda) \subset W^{16}, \text{ and similarly} \\
f_\alpha^{-1}f_\beta \MCG(X_2) (f_\alpha^{-1}f_\beta)^{-1} \subset \MCG(A_\Lambda) \subset W^{16}.
\end{gather*}
It follows that $\MCG(X_1)$, $\MCG(X_2) \subset W^{20}$, so $\MCG(X_1 \cup X_2) \subset W^{40}$. Now we can complete the proof by repeating the above argument to the union of the odd $M_n$.
\end{proof}

We can strengthen Lemma \ref{cofinitelem} using the upgraded Sierpi\'nski lemma.

\begin{lemma} \label{cofinitelem2}
Let $S$ be any surface containing a countable admissible family of disjoint subsurfaces $\{S_n\}$ each homeomorphic to a sliced Loch Ness monster, and $W \subset \PMCGcc{S}$ a countably syndetic symmetric set. Let $\{M_{n,m}\}_{m \in \Z}$ be the collection of standard pieces for $S_n$. Then there is a finite set $F$ such that $$\displaystyle\prod_{(n,m) \notin F} \MCG(M_{n,m}) \subset W^{80}$$
\end{lemma}

\begin{proof}
We show the case where $\{S_n\}_{n \in \N}$ is infinite since the finite case is similar. Apply Lemma \ref{techlem} with $\mathcal{A}$ consisting of subsurfaces of the form $A_{\Lambda} = \bigcup_{n \in \N, m \in \Lambda}  M_{n,2m}$ to show $\MCG(A_{\Lambda}) \subset W^{16}$ for some infinite $\Lambda \subset \Z$. Apply Lemma \ref{sierpinski2} to $\N \times \Lambda$, so that $\{\Omega_\alpha\}_{\alpha \in \R}$ are infinite subsets of $\N \times \Lambda$ with the properties listed in the lemma. Note we may assume $\N \times \Lambda$ and each $\Omega_\alpha$ contain infinitely many positive and negative integers in the second coordinate.

Modify all of the standard pieces slightly as before, then for each $\alpha$, let $f_\alpha$ be a map supported in $\bigcup_{i \in \N} S_i$ with the following property. For all $(n,m) \in \Omega_\alpha$, let $f_\alpha(M_{n,2m})$ be the smallest connected subsurface containing the union of $M_{n,2m+1}, M_{n,2m+2}$,..., $M_{n,2k-1}$ where $k$ is the smallest second component among the elements of $\Omega_\alpha \cap (\{n\} \times \Lambda)$ larger than $m$. Now the proof is completed as before.
\end{proof}

\subsection{Fragmentation} \label{fragmentation}

To work with the subgroup $\overline{\PMCG_c(S)}$, we will use results of the author \cite{Dickmann_2023} for decomposing  an element of $\overline{\PMCG_c(S)}$ into simpler pieces. 

\begin{lemma} \label{fraglemma}
    Let $S$ be any infinite-type surface and $f \in \overline{\PMCG_c(S)}$. There exist two sequences of compact subsurfaces $\{K_{i}\}$ and $\{C_{i}\}$, with each sequence consisting of  disjoint surfaces, and $g,h \in \overline{\PMCGc(S)}$ such that 
    \begin{enumerate}[(i)]
        \item
            $\supp(g) \subseteq \bigcup_{i} C_{i}$ and $\supp(h) \subseteq \bigcup_{i} K_{i}$,
        \item
            $f=hg$.
    \end{enumerate}

Furthermore, if $S$ is a disk with handles we can assume the following: 

\begin{enumerate} [(i)]
\item Each $\partial K_i$ and $\partial C_i$ is a single essential simple closed curve.

\item $S \searrow \cup_i K_i$ and $S \searrow \cup_i C_i$ are homeomorphic to $S$ with compact boundary components added accumulating to some subset of the ends.
\end{enumerate} 

\end{lemma}

We use $S \searrow K$ or $S_K$ to denote the surface obtained from cutting $S$ along $K$. Similar to the definition of cutting along a curve or arc, $S_K$ is a surface with boundary with an identification map from some subset of $\partial S_K$ to some subset of $\partial K$ such that the quotient on $S_K \sqcup K$ is homeomorphic to $S$. Note we can realize $S_K$ as a subsurface of $S$, in particular the closure of the complement of $K$.

%%%%%%%%%%%%%
%%%%%%%%%%%%%%%%%%%
\section{Proof of Main Results} \label{mainsection}

Now we are ready to prove the results from the introduction. First, we prove the reverse directions of the main theorems as well as Theorem \ref{thm1.3} using the tools from the previous sections. Then we prove the forward directions using the work Domat \cite{domat2020big} and Lemma \ref{finitechain}. We prove Corollary \ref{cor1.1} and Theorem \ref{thm1.2} afterward.  

\subsection{Reverse Directions of Theorems \hyperref[thma]{A} and \hyperref[thmb]{B}} \label{reverse}

First, we will prove the reverse direction of Theorem \hyperref[thmb]{B}, and then we will prove the reverse direction of Theorem \hyperref[thma]{A} and Theorem \ref{thm1.3} with a similar method. Recall the Steinhaus property implies automatic continuity by Proposition \ref{rosendal}. Let $\operatorname{Stab}(K)$ denote the subgroup consisting of maps that pointwise fix a subsurface $K$. When $K$ is finite-type, $\operatorname{Stab}(K)$ is an open neighborhood of the identity, and the collection of all such stabilizers is a neighborhood basis of the identity.

\begin{proposition} \label{reverse_b}
Let $S$ be a connected sum of a finite-type surface with finitely many disks with handles. Then $\PMCGcc{S}$ is Steinhaus with constant 328.
\end{proposition}

 \begin{proof}

    First, we explain the details for the sliced Loch Ness monster and then discuss how to extend the argument to the other cases.
 
\p{Case 1: sliced Loch Ness monsters} Suppose $S$ is any sliced Loch Ness monster. Let $W$ be any countably syndetic symmetric subset of $\MCG(S)$. By Proposition \ref{bct}, let $U$ be an open neighborhood of the identity such that $W^2$ is dense in $U$, and find some compact subsurface $K$ such that $\operatorname{Stab}(K)\subseteq U$. Note that any sliced Loch Ness monster $S$ has a compact exhaustion $\{K_i\}$ where each $S \searrow K_i$ is homeomorphic to $S$ (see Figure \ref{SLNM}). Thus we can assume that $S_K = S \searrow K$ is homeomorphic to $S$ by replacing $K$ with a large enough $K_i$ if needed. Now we want to show that $\MCG(S_K) = \operatorname{Stab}(K) \subseteq W^k$ for some $k$.

Let $f \in \MCG(S_K)$ be any element. Let $g$ be one of the maps produced by applying Lemma \ref{fraglemma} to $f$, and assume the conditions of the second part of Lemma \ref{fraglemma} hold for $g$. Note here we are using Lemma \ref{structure} which implies $\MCG(S_K) = \PMCGcc{S_K}$. Let $\{K_i\}$ be the sequence of compact subsurfaces containing the support of $g$. Let $\{M_i\}$ be the collection of standard pieces for $S_K$. We can assume by applying change of coordinates that each $M_i$ contains exactly one of the $K_i$, and each $K_i$ appears in some $M_i$. By Lemma \ref{cofinitelem}, we have some cofinite union $T = \bigcup_{i \in \Z \setminus F} M_i$ with $\MCG(T) \subset W^{80}$. Therefore, we can find some $g^\prime \in \MCG(T)$ such that $g^\prime g \in \PMCG_c(S_K)$. Now let $K^\prime \subset S_K$ be a compact subsurface bounded by a single curve that contains the support of $g^\prime g$. Now using the density of $W^2$ in $\MCG(S_K)$, find some element $\phi \in W^2$ such that $\phi(K^\prime) \subset T$. 
It follows that $\phi g^\prime g \phi^{-1} \in W^{80}$, and therefore $g^\prime g \in W^{84}$. Finally, this gives $g \in W^{164}$, and since the above argument also applies to the other element from fragmentation, $f \in W^{328}$.  

\p{Case 2: disks with handles}
Now assume $S$ is a disk with handles, and let $W$ be a countably syndetic symmetric subset of $\overline{\PMCG_c(S)}$. Let $U$ be an open neighborhood of the identity such that $W^2$ is dense in $U$, and find some compact subsurface $K$ such that $\operatorname{Stab}(K)\subseteq U$. We now claim we can enlarge $K$ so that each component of $S_K = S \searrow K$ has exactly one boundary chain and no compact boundary components. First note that there exists a compact exhaustion $\{K_i\}$ of $S$ such that each $\partial K_i$ is a single component that intersects $\partial S$ in a union of closed intervals, and the components of $S \searrow K_i$ are infinite-type without compact boundary components. To build such an exhaustion, start with the disk with boundary points removed used to construct $S$, call it $D$, and find a compact exhaustion $\{C_i\}$ of $D$ such that each $C_i$ is a disk and each $D \searrow C_i$ is a union of disks with boundary points removed. Then we get the desired exhaustion on $S$ by attaching handles to $D$ and modifying the $C_i$ accordingly. During this last step, we must require the attaching regions for any handle to be disjoint from each $\partial C_i$ and that whenever one attaching region of some handle is contained in the interior of $C_i$, then then other attaching region is also contained in the interior. Now note each $S \searrow K_i$ has a single boundary chain since $S$ has one boundary chain and all of the boundary components of $S \searrow K_i$ point to boundary ends of $S$. The claim follows by replacing $K$ with a large enough $K_i$ if needed. Now since $S_K$ is also infinite-type and has no interior ends, Lemma \ref{finitechain} implies $S_K$ is a disjoint union of disks with handles. The complement of any compact subsurface necessarily has a finite number of components, so $S_K$ is a finite disjoint union.

Let $f \in \MCG(S_K)$ be any element. Let $g$ be one of the maps produced by applying fragmentation to $f$, and assume the conditions of the second part of Lemma \ref{fraglemma} hold for $g$ when restricted to each component of $S_K$. In this case, we need to apply Lemma \ref{fraglemma} to each component separately and then combine. Let $\{K_i\}$ be the collection of compact subsurfaces containing the support of $g$. Using Lemma \ref{fulldecomp}, we can cut $S_K$ along a collection of disjoint arcs into sliced Loch Ness monsters $\{S_n\}$. We can also assume these arcs are chosen to miss the $K_i$. 

Let $\{M_{n,m}\}$ be the collection of standard pieces for $S_n$. By change of coordinates, we can assume each $M_{n,m}$ contains exactly one $K_i$, and each $K_i$ appears in some $M_{n,m}$. By Lemma \ref{cofinitelem2}, there is a finite set $F$ such that $T = \bigcup_{(n,m) \notin F} M_{n,m}$ and $\MCG(T) \subset W^{80}$. Proceed as before.

\p{Case 3: connected sums}
We now need to consider the general case when $S$ is a connected sum of a finite-type surface with finitely many disks with handles. Let $W$ be a countably syndetic symmetric subset of $\overline{\PMCG_c(S)}$. Let $U$ be an open neighborhood of the identity such that $W^2$ is dense in $U$, and find some finite-type subsurface $K$ such that $\operatorname{Stab}(K)\subseteq U$.  By choosing $K$ large enough to contain all the punctures and compact boundary components, we can ensure each component of $S \searrow K$ has one boundary chain, no interior ends, and no compact boundary components. By applying Lemma \ref{finitechain}, we have that $S \searrow K$ is homeomorphic to a disjoint union of finitely many disks with handles.  Proceed as in the disk with handles case.
\end{proof}

Now we show the reverse direction of Theorem \hyperref[thma]{A} using similar techniques. This will also prove Theorem \ref{thm1.3}, the extension to the full mapping class group.

\begin{proposition} \label{reverse_a}
Let $S$ be a surface satisfying the conditions of Proposition \ref{reverse_b} and with finitely many ends accumulated by genus. Then $\PMCG(S)$ and $\MCG(S)$ are Steinhaus with constant 328.
\end{proposition}

 \begin{proof} Let $W$ be a countably syndetic symmetric subset of $\PMCG(S)$. Let $U$ be an open neighborhood of the identity such that $W^2$ is dense in $U$, and find some compact subsurface $K$ such that $\operatorname{Stab}(K)\subseteq U$. As before we can enlarge $K$ if needed so that $S \searrow K$ is homeomorphic to a disjoint union of finitely many disks with handles. Since there are finitely many ends accumulated by genus, we can further assume that each component of $S \searrow K$ is a sliced Loch Ness monster. This ensures that $\operatorname{Stab}(K) \subset \PMCGcc{S}$ so we can use fragmentation as in the proof of Proposition \ref{reverse_b} to show $\operatorname{Stab}(K) 
 \subset W^{328}$. The proof for $\MCG(S)$ is identical. Note that we must also use a minor adaptation of Lemma \ref{techlem} where $W$ is a countably syndetic symmetric subset of $\PMCG(S)$ or $\MCG(S)$ instead of $\PMCGcc{S}$.
 \end{proof}

\subsection{Forward Directions of Theorems \hyperref[thma]{A} and \hyperref[thmb]{B}} \label{forward}

Now we finish the proofs of Theorems \hyperref[thma]{A} and \hyperref[thmb]{B} using the work of Domat \cite{domat2020big}. A \textit{nondisplaceable} surface in $S$ will refer to a subsurface $K$ disjoint from the noncompact boundary components of $S$ such that $f(K) \cap K \ne \emptyset$ for all representatives of $f \in \PMCGcc{S}$. Note a subsurface $K$ is nondisplaceable if it separates ends; i.e., if $S \searrow K$ is disconnected and induces a partition of $E(S)$ into two or more sets. A subsurface is also nondisplaceable if it separates the ends of the interior surface, so a subsurface that separates boundary components or separates a boundary component from an end is also nondisplaceable. The following result is implicit from Sections 6,7,8, and 10 of Domat's paper.

\begin{lemma} [Domat] \label{domat}
 Let $S$ be an infinite-type surface such that either
 
 \begin{enumerate} [(i)]
     \item $S$ has genus at least 3 and there exists an infinite sequence of disjoint nondisplaceable essential annuli that eventually leaves every compact subsurface.
     \item $S$ has any genus and there exists an infinite sequence of disjoint nondisplaceable essential spheres with n punctures and b boundary components for $n + b \geq 8$, and the sequence eventually leaves every compact subsurface.
 \end{enumerate}

Then there exists a discontinuous homomorphism $\phi: \PMCGcc{S} \rightarrow \Q$. 
\end{lemma}

Although Domat's work focused on surfaces with compact boundary, the conditions in Lemma \ref{domat} hold for some surfaces with noncompact boundary, and the proof goes through without adaptation. For surfaces with only compact boundary components, Domat showed the first condition holds when the interior of $S$ has at least two ends and at least one end accumulated by genus and the second condition holds when the interior has infinitely many ends. In the case of the Loch Ness monster, there are no finite-type nondisplaceable subsurfaces, so this case had to be handled separately. Domat and the author showed in the appendix of Domat's paper that $\MCG(S)$ does not have automatic continuity when $S$ is a Loch Ness monster.

\begin{proof}  [Proof of Theorem {\hyperref[thmb]{B}}]
Recall the reverse direction was shown in Proposition \ref{reverse_b}. When $S$ either has infinitely many interior planar ends, infinitely many compact boundary components, infinitely many boundary chains, or at least one interior end accumulated by genus we will show one of the conditions of Lemma $\ref{domat}$ is satisfied so that $\PMCGcc{S}$ does not have automatic continuity. After possibly filling in the finite number of punctures and capping the finite number of compact boundary components, we apply Lemma \ref{finitechain} to conclude $S$ is a connected sum of finitely many disks with handles. The original $S$ can then be obtained by connect summing with a finite-type surface with punctures and boundary components. 

\p{Case 1: infinitely many interior planar ends}  When $S$ has infinitely many planar interior ends, there is a closed neighborhood $U$ of one of these ends such that each component of $\partial U$ is compact, and $U$ has infinitely many planar ends. Now this case reduces to the cases originally considered by Domat, and the second condition of Lemma \ref{domat} holds.  
 
\p{Case 2: infinitely many compact boundary components}  When $S$ has infinitely many compact boundary components, there is some end accumulated by compact boundary, and every closed neighborhood of this end contains infinitely many compact boundary components. Let $\{U_i\}$ be a system of closed neighborhoods of this end such that $U_{i+1} \subset U_i$ for all $i$, and $\bigcap_{i=1}^\infty U_i = \emptyset$. Now the second condition of Lemma \ref{domat} holds by inductively choosing an essential punctured sphere in some sufficiently small $U_i$ that misses the previously chosen punctured spheres.

\p{Case 3: infinitely many boundary chains}  When $S$ has infinitely many boundary chains, there is some end such that every closed neighborhood has infinitely many boundary chains. Let $\{U_i\}$ be a system of closed neighborhoods of this end satisfying the two properties from the previous case. Since the interior of each $U_i$ has infinitely many ends, we can use an inductive procedure as before to show the second condition of Lemma \ref{domat} holds. 

\p{Case 4: at least one interior end accumulated by genus} When $S$ has an interior end accumulated by genus, there is a closed neighborhood $U$ of this end such that each component of $\partial U$ is compact, and $U$ has infinite genus. We can assume $S$ is not the Loch Ness monster since this case was ruled out by Domat and the author, so the interior of $S$ has at least two ends. Now this case reduces to the cases considered by Domat, and the first condition of Lemma \ref{domat} holds. The interior of $S$ having at least two ends ensures that we can find annuli that separate ends of the interior.
\end{proof}

\begin{proof} [Proof of Theorem {\hyperref[thma]{A}}]
    
Recall the reverse direction was shown in Proposition \ref{reverse_a}. Now we consider two cases, and then we are done by using Lemma \ref{finitechain} as in Theorem \hyperref[thmb]{B}.

 \p{Case 1: infinitely many ends accumulated by genus} First, suppose $S$ has infinitely many ends accumulated by genus. Let $\phi_1: \PMCG(S) \rightarrow \Z^\omega$ be the projection mapping given by Lemma \ref{structure}, and let $\phi_2: \Z^\omega \rightarrow (\Z_2)^\omega$ be the mod 2 homomorphism. Now we use the discontinuous homomorphism $\psi: (\Z_2)^\omega \rightarrow \Z_2$ from Example 1.4 of Rosendal \cite{Rosendal09}. Composing all of the above homomorphisms yields a discontinuous homomorphism $\psi\circ\phi_2\circ\phi_1: \PMCG(S) \rightarrow \Z_2$. To show that $\psi\circ\phi_2\circ\phi_1$ is discontinuous we use the discontinuity of $\psi$ and the fact that the $\phi_i$ are surjective and thus open by the open mapping theorem for Polish groups.
 
 \p{Case 2: finitely many ends accumulated by genus}
 Now suppose $S$ has finitely many ends accumulated by genus and satisfies one of the conditions for Lemma \ref{domat}. Note $\PMCGcc{S}$ is a countable index subgroup of $\PMCG(S)$ by Lemma \ref{structure}. By Lemma \ref{domat}, there is a map $\phi: \PMCGcc{S} \rightarrow \Q$ such that $\ker(\phi)$ is not open. Note we also have that $\ker(\phi)$ is not open in $\PMCG(S)$. Since $\ker(\phi)$ is countable index in $\PMCGcc{S}$, it must also be countable index in $\PMCG(S)$. Now $\PMCG(S)$ cannot have automatic continuity by Proposition \ref{smallindex}. 
\end{proof}

\subsection{Consequences} \label{extra}

 Now we finish with the proofs of some additional results.
  
  \begin{proof} [Proof of Theorem \ref{thm1.2}]
  
  Recall the Loch Ness monster case was shown by Domat and the author \cite{domat2020big}, so we can assume the interior of $S$ has at least two ends. By Lemma \ref{domat}, $\PMCGcc{S}$ has a nonopen countable index subgroup. Since $\PMCG(S)$ is finite index in $\MCG(S)$ by the assumption of finitely many ends, we can apply the same proof as the second case of Theorem \hyperref[thma]{A} to show $\MCG(S)$ does not have automatic continuity.
\end{proof}

\begin{proof} [Proof of Corollary \ref{cor1.1}] Suppose $S$ is a disk with handles. Proceeding by contradiction, let $H$ be the kernel of a nontrivial map from $\PMCGcc{S}$ to a countable group with the discrete topology. By automatic continuity, $H$ is open and closed. Since $H$ is closed, it suffices to show that it contains every compactly supported mapping class since then $H$ is dense in $\PMCGcc{S}$, and in fact $H = \PMCGcc{S}$. Since $H$ is open, it contains $\Stab(K)$ for some compact subsurface $K$. Now let $\phi$ be any compactly supported mapping class. Since $\phi$ fixes the boundary pointwise, we can isotope it so that it is supported in a subsurface $K^\prime$ that does not intersect the boundary. Since $S$ is a disk with handles, its interior is a Loch Ness monster, and there exists some homeomorphism supported in the interior that takes $K^\prime$ into the complement of $K$. Therefore, a conjugate of $\phi$ lies in $\Stab(K) \subset  H$, and we are done since $H$ is normal. There are no proper finite index subgroups in $\PMCGcc{S}$ since any index $n$ subgroup determines a nontrivial homomorphism to the symmetric group on $n$ elements via the left multiplication action on the left cosets.
\end{proof}

\end{document}